\documentclass[12pt]{amsart}
\usepackage{amssymb}
\usepackage{verbatim}
\usepackage{amscd}

\theoremstyle{plain}
\newtheorem{theorem}{Theorem}
\newtheorem{corollary}{Corollary}
\newtheorem{proposition}{Proposition}
\newtheorem{lemma}{Lemma}

\theoremstyle{definition}
\newtheorem{definition}{Definition}

\theoremstyle{remark}
\newtheorem{remark}{Remark}
\newtheorem{example}{Example}

\def\bbN{\mathbb N}

\def\bbR{\mathbb R}
\def\bbT{\mathbb T}
\def\bbZ{\mathbb Z}

     \newcommand{\sB}{\mathcal B}
     
     \newcommand{\sD}{\mathcal D}

     \newcommand{\sG}{\mathcal G}

     \newcommand{\sL}{\mathcal L}
     
     \newcommand{\sN}{\mathcal N}
     \newcommand{\sO}{\mathcal O}
     \newcommand{\sP}{\mathcal P}

     \newcommand{\sU}{\mathcal U}
     \newcommand{\sV}{\mathcal V}
     \newcommand{\sW}{\mathcal W}
     
     \newcommand{\sY}{\mathcal Y}

\newcommand{\al}{\alpha}

\newcommand{\ep}{\epsilon}
\newcommand{\vpi}{\varphi}

\newcommand{\de}{\delta}
\newcommand{\om}{\omega}

\begin{document}

% topmatter
\title[Semigroups \& transfer operators]{Semigroups of locally injective maps and transfer operators}

\author{Justin R. Peters}
\address{Department of Mathematics\\
    Iowa State University, Ames, Iowa, USA}

\email{peters@iastate.edu}
\thanks{The author acknowledges partial support from the National Science Foundation, DMS-0750986}

% \keywords{keywords}
% \thanks{2000 {\itshape Mathematics Subject Classification}.
%  Primary, *****; Secondary, *****.}
% \date{\date}

% abstract
 \begin{abstract} We consider semigroups of continuous, surjective, locally
 injective maps of a compact metric space, and whether such
 semigroups admit a transfer operator.
 \end{abstract}
\maketitle

\section{Introduction}
In this note we consider certain semigroups of continuous,
surjective, locally injective maps acting on a compact metric space.
In \cite{ER} R. Exel and J. Renault looked at crossed products
arising from semigroups of local homeomorphisms acting on a compact
metric space. In particular, the semigroups studied were assumed to
satisfy an admissibility condition. Admissibility is equivalent to
the existence of a transfer operator. The object here is to examine
the question of admissibility of a slightly broader class of maps.

If $\sP$ is a semigroup and $\vpi: \sP \to CSLI(X),\ n \to \vpi_n,$
is an isomorphism of $\sP$ into the continuous, surjective, locally
injective maps of a compact metric space $X$ to itself, then $\sP$
can be viewed as a semigroup of unital endomorphisms $\al_n$ of
$C(X)$, via $\al_n(f) = f\circ\vpi_n.$ A \emph{transfer operator} is
a linear map $\sL_n: C(X) \to C(X)$ which is a left inverse of
$\al_n,\ n \in \sP.$ Transfer operators have been studied in the
contexts of both ergodic and topological dynamical systems
(cf~\cite{Ba}). Here the maps we consider are locally injective, so
that a transfer operator, if it exists, is of the form
\[ \sL_n(f)(y) = \sum_{\vpi_n(x) = y} \om(n, x) f(x) \]
where $\om(n, \cdot)$ is a cocycle for the semigroup $\sP.$ Local
injectivity is neither necessary nor sufficient for the existence of
a transfer operator.

Unlike the situation for a local homeomorphism, where a cocycle,
hence a transfer operator, can always be defined, for continuous,
surjective, locally injective (CSLI) maps, it  may happen that no
cocycle exists, even if one relaxes the condition of strict
positivity.  Indeed, as we show, the existence of a strictly
positive cocycle for a CSLI map $\vpi$ implies that $\vpi$ is a
local homeomorphism.  In our context, an admissible dynamical system
is one which admits a nonnegative, continuous cocycle. (See
Definition~\ref{d:cocycle}.) As we show, this weaker form of
admissibility is not always satisfied, even in the case of a single
CSLI map, that is where the semigroup is isomorphic with $\bbN.$
(See Example~\ref{e:notadmiss}.) In the case of semigroups with a
finite set of free generators (i.e., isomorphic to $(\bbZ^+)^k$), we
give necessary and sufficient conditions on the generators for
admissibility.

We also consider another class of semigroups, divisible semigroups,
which include examples such as the (additive) semigroup of positive
dyadic rationals.  There we give necessary (but not sufficient)
conditions for a semigroup to be admissible, and an example where
the cocycle is constructed.

I would like to thank the referee for his suggestions.

\section{Preliminaries}
All actions will take place on a compact metric space $X$. The class
of mappings studied here are continuous, surjective, and locally
injective, which we abbreviate as CSLI. Note that if a CSLI map
$\vpi$ is also open, then it is a local homeomorphism. The class of
CSLI maps on a metric space $X$ is hereditary in the sense that if
$Y \subset X$ is a closed subset such that $\vpi(Y) = Y,$ then
$\vpi$ is also a CSLI map of $Y.$  By contrast, local homeomorphisms
are not hereditary.

We will use $\sP$ to denote a semigroup. For the semigroups $\sP$
considered there is a group $\sG$ such that $\sP \subset \sG$, and
the group operation on $\sG$ maps $\sP \times \sP \to \sP.$  The
semigroups considered here will be abelian, and the group operation
will be written additively.

Thus, we are given a compact metric space $X$, an abelian semigroup
$\sP,$ and a map $\vpi$ of $\sP$ into the CSLI maps of $X,$
satisfying
\[ \vpi_{n+m} = \vpi_n \circ \vpi_m \ \text{ for all } n,\ m \in
\sP.\] We will refer to the pair $(\sP, X)$ as a \emph{dynamical
system}. In case the semigroup $\sP$ is the natural numbers $\bbN,$
there is a generator $\vpi_1;$ we will write $\vpi$ in place of
$\vpi_1.$

If the semigroup $\sP$ is contained in a torsion group $\sG$, so
that for each $n \in \sP$ there is $k \in \bbN$ such that $kn = 0,$
then $\vpi_n$ is a homeomorphism, since the $k-$fold composition
with itself is a homeomorphism.  Such semigroups are not of interest
here, where we want to study CSLI maps which are not homeomorphisms.
The semigroups $\sP$ will not necessarily be assumed to contain the
zero element, in which case the corresponding subset $\{ \vpi_n:\ n
\in \sP\}$ will not contain the identity.

\begin{definition} \label{d:cocycle} We will say $\om$ is a \emph{cocycle} on a dynamical system $(\sP,
X)$ if
\begin{enumerate}
\item $\om$ is a function from $\sP \times X \to \bbR,$ and $\om(n,
x) \geq 0 $ for all $(n, x) \in \sP \times X;$
\item for each $y \in X,\ n \in \sP,\ \sum_{\vpi_n(x) = y} \om(n, x)
= 1;$
\item for each $n \in \sP,$ the map $ x \in X \to \om(n, x)$ is
continuous;
\item $\om$ satisfies the cocycle identity:
\[ \om(m+n, x) = \om(m, x) \om(n, \vpi_m(x)) .\]
\end{enumerate}
A dynamical system $(\sP, X)$ will be called \emph{admissible} if it
admits a cocycle.
\end{definition}
Our definition of admissibility differs from that in \cite{ER} as we
do not include the requirements of strict positivity or coherence in
the definition of admissibility.  Indeed, for a singly generated
semigroup, say with generator a CSLI map $\vpi,$ we show there
exists a strictly positive cocycle for the action if and only if
$\vpi$ is a local homeomorphism (Corollary~\ref{c:strictlypos}).

In the case of singly generated semigroups $\sP$ (isomorphic with
$\bbN$), admissibility depends on the existence of a cocycle for the
generator, denoted by $\vpi.$ Indeed, if a cocycle exists for $\vpi$
the cocycle for $\vpi^n$ (the $n$-fold composition of $\vpi$) is
then determined by the cocycle identity (4).  Thus, to simplify
notation, for singly generated semigroups we write the cocycle as
$\om(\cdot)$ and omit the dependence on the semigroup.

If $F$ is a finite or infinite set, we use $|F|$ to denote the
cardinality of $F.$ We will denote the metric on the space $X$ by
$\rho$.

\section{Semigroups of CSLI maps}
\subsection{Generalities concerning CSLI maps}
We begin with some elementary topological results for CSLI maps.

\begin{lemma} \label{l:bddpoint} Let $(\vpi, X)$ be a CSLI dynamical system.  Then for
all $x \in  X, \ |\vpi^{-1}(x)| < \infty.$
\end{lemma}

\begin{proof}  If  for some $x_0,\ |\vpi^{-1}(x_0)|$ is infinite,
there is a sequence $\{u_n\} \subset \vpi^{-1}(x_0),$ with the $u_n$
distinct points, which converges to a point $u_0$. Thus, $\vpi(u_n)
\to \vpi(u_0).$ As $\vpi(u_n) = x_0,$ it follows $\vpi(u_0) = x_0.$
But then $\vpi$ is not injective on any neighborhood of $u_0.$
\end{proof}

\begin{lemma} \label{l:card} Let $(\vpi, X)$ be CSLI, $y_0 \in  X, \ \vpi^{-1}(y_0)
= \{x_1, \dots, x_N\}.$ Given $\ep > 0$ let the compact
neighborhoods $\sN_j = \{ x: \rho(x, x_j) \leq \ep \}$. Then there
is a $\de
> 0$ so that if $\sU = \{ y: \rho(y, y_0) \leq \de\}$ then
$\vpi^{-1}(\sU) \subset \cup_{j=1}^N \sN_j.$
\end{lemma}

\begin{proof}  Assume that no such $\de $ exists, and let $\sU_n$ be a nested neighborhood base at $y_0$ and
$x_n \in \vpi^{-1}(\sU_n),\ x_n \notin \cup_{j=1}^N \sN_j.$ Taking a
subsequence, we may assume $x_n \to x_0',$ for some point $x_0' \in
X.$ Since $\vpi(x_n) \in \sU_n,$ \ $\vpi(x_n) \to y_0.$  Hence $
\vpi(x_0') = y_0.$ But that is impossible, as $\rho(x_0', x_j) \geq
\ep$ for $j = 1, \dots, N.$
\end{proof}

\begin{remark} \label{r:disjoint} Note that $\ep > 0$ can be taken
sufficiently small so that the $\sN_j$ are pairwise disjoint, and so
that $\vpi$ is injective on $\sN_j.$
\end{remark}

\begin{corollary} \label{c:cardopen} The set $\{ u \in X: \ |\vpi^{-1}(u)| \leq N \}$ is open.
\end{corollary}

\begin{proof} The neighborhoods $\sU,\ \sN_j$ of Lemma~\ref{l:card} can be
made small enough so that $\vpi$ is injective on each $\sN_j.$ Thus,
for any $u \in \sU,\ |\vpi^{-1}(u)| \leq N.$
\end{proof}

\begin{corollary} \label{c:bddinv} There exists $N \in \bbN$ such
that $\sup_{y \in X} |\vpi^{-1}(y)| = N.$
\end{corollary}

\begin{proof} By Corollary~\ref{c:cardopen} for each $y \in X$ there
is a neighborhood $\sU_y$ of $y$ and a minimal integer $N_y$ so that
for $u \in \sU_y,\ |\vpi^{-1}(u)| \leq N_y.$ Now by compactness of
$X$ there is a finite subcover $\sU_{y_i}$, and $N$ can be taken as
the the maximum of the corresponding $N_{y_i}.$
\end{proof}

The next lemma is known, but we include it here for completeness.

\begin{lemma} \label{l:open} Let $\vpi$ be a local homeomorphism of a compact
metric space $X$.  Let $y_0 \in X$ and suppose $|\vpi^{-1}(y_0)| =
N.$  Then there is an open neighborhood $\sU$ of $y_0$ for which
$|\vpi^{-1}(y)| \geq N, \ y \in \sU.$
\end{lemma}

\begin{proof} Let $\vpi^{-1}(y_0) = \{ x^0_{(j)}:\ j = 1, \dots, N\}.$
Let $\sN_j$ be a neighborhood of $x^0_{(j)}$ such that $\vpi|\sN_j$
is a homeomorphism, and $\sN_i \cap \sN_j = \emptyset,$ for $ i \neq
j.$ Let $\sU_j = \vpi(\sN_j),\ j = 1, \dots, N,$ and set
\[ \sU = \cap_{j=1}^N \sU_j, \text{ and } \sW_j = \vpi^{-1}(\sU)\cap \sN_j.\]
Set $\vpi^{(j)} = \vpi|\sW_j.$

Let $y \in \sU$ and $x_{(j)} = {\vpi^{(j)}}^{-1}(y),\ j = 1, \dots,
N.$ Then $\{ x_{(j)}:\ j = 1, \dots, N\}$ is a set of cardinality
$N$, so that $|\vpi^{-1}(y)| \geq N.$
\end{proof}

\begin{corollary} \label{c:clopen} Let $(\vpi, X)$ be a CSLI system,
and let $X_j = \{y \in X: \ |\vpi^{-1}(y)| = j\}.$ Then $\vpi$ is a
local homeomorphism if and only if each $X_j$ is both closed and
open.
\end{corollary}

\begin{proof} If $\vpi$ is a local homeomorphism, it follows from
Corollary~\ref{c:cardopen} and Lemma~\ref{l:open} that each $X_j$ is
open. But then $\cup_{i\neq j} X_i$ is open, so $X_j$ is also
closed.

Conversely suppose that each $X_j$ is clopen. Then, referring to the
proofs of Corollary~\ref{c:bddinv} and Lemma~\ref{l:card}, we may
choose a sufficiently small neighborhood $\sU$ of $y_0$ so that $\sU
\subset X_N.$ Then each point $y$ of $\sU$ has exactly $N$ inverse
images. We can choose $\ep$ sufficiently small so that $\vpi|\sN_j$
is injective. Set $\sW_j = \vpi^{-1}(\sU)\cap \sN_j.$ Then
$\vpi|\sW_j$ is a homeomorphism from $\sW_j$ onto $\sU,\ j = 1,
\dots, N.$ Thus, $\vpi$ is a local homeomorphism.
\end{proof}

\begin{remark} \label{r:open} Suppose $(\vpi, X)$ is a CSLI system, $y_0 \in X,\
|\vpi^{-1}(y_0)| = N,$ and  $\sU$ is a compact neighborhood of $y_0$
such that  $\vpi^{-1}(\sU) = \cup \sN_j$ where the $\sN_j$ are
pairwise disjoint, and $\vpi|\sN_j$ is injective, $j = 1, \dots, N$.
It does not follow that $\vpi|\sN_j$ is a homeomorphism of $\sN_j$
onto $\sU.$  Generally, $\vpi|\sN_j$ will be onto a proper subset of
$\sU.$ The Baire Category Theorem asserts that for some $j,\
\vpi(\sN_j)$ will have nonempty interior, but there is no reason
that the point $y_0$ must lie in the interior.
\end{remark}

\subsection{Admissible CSLI systems}
\begin{definition} Let $(\vpi, X)$ be a CSLI dynamical system, and
$x_1 \in X.$ We say $\vpi$ is \emph{locally open} at $x_1$ if there
is an open neighborhood $\sN$ of $x_1$ such that the restriction
$\vpi|\sN$ is an open map of $\sN $ into $X.$
\end{definition}

\begin{lemma} \label{l:locopenadmiss} Suppose $(\vpi, X)$ is a
CSLI dynamical system, $y_0 \in X,$ \ \mbox{$\vpi^{-1}(y_0) = \{
x_1, \dots x_N\}$}.  Suppose the system admits a cocycle $\om$ and
$\om(x_1) > 0.$  Then $\vpi$ is locally open at $x_1.$
\end{lemma}

\begin{proof} We can take a compact neighborhood $\sU$ of $y_0$ so
that $\vpi^{-1}(\sU) = \cup_{i=1}^N \sW_i$ where $x_i \in \sW_i, \ 1
\leq i \leq N$ and the $\sW_i$ are pairwise disjoint (and compact),
and so that the restriction of $\vpi$ to $\sW_i$ is injective. We
can also assume, by taking $\sU$ sufficiently small and applying
Lemma~\ref{l:card}, that $\eta := \min\{ \om(x):\ x \in \sW_1\} >
0,$ and for $x \in \sW_i,\ |\om(x) - \om(x_i)| < \frac{\eta}{N},\ 1
\leq i \leq N.$

We claim $\vpi(\sW_1) = \sU.$ If not, there is some $y \in \sU$ such
that $\vpi^{-1}(y) = \{x_i': \ x_i' \in \sW_i,\ i \in I\}$ where $I$
is a subset of $\{ 2, \dots, N\}.$ But then
\begin{align*}
\sum_{\vpi(x) = y} \om(x) &= \sum_{i \in I} \om(x_i') \\
    &< \sum_{i \in I} [ \om(x_i) + \frac{\eta}{N}] \\
    &< \sum_{i \in I} \om(x_i) + \eta \\
    &\leq 1
\end{align*}
contradicting the cocycle property. Thus the claim is verified.

Let $\sV$ be an open set in $X$ containing $\sW_1$ and disjoint from
$\sW_2, \dots, \sW_N.$ Let $\sU^o$ be the interior of $\sU.$  Now
$\vpi^{-1}(\sU^o)\cap \sV =: \sY_1$ is an open neighborhood of $x_1$
disjoint from $\sW_2, \dots, \sW_N,$ so is contained in $\sW_1,$ and
hence in the interior $\sW_1^o.$  We claim that the restriction
$\vpi|\sY_1$ is an open map of $\sY_1$ into $X.$  Let $\sO \subset
\sY_1$ be open.  Since $\vpi|\sW_1$ is a one to one continuous map
of $\sW_1$ onto $\sU,$ it is a homeomorphism.  Thus, $\vpi(\sO)$ is
open in $\sU.$  But $\vpi(\sO) \subset \sU^o,$ so that $\vpi(\sO)$
is open in $X$. Hence, $\vpi$ is locally open at $x_1.$
\end{proof}

\begin{corollary} \label{c:locopenadmiss} A necessary condition for a CSLI system $(X,
\vpi)$ to admit a cocycle is: for every $y \in X$ there exists a
point $x \in \vpi^{-1}(y)$ such that $\vpi$ is locally open at $x.$
\end{corollary}

\begin{corollary} \label{c:strictlypos} Let $(\vpi, X)$ be a CSLI
system.  Then the system admits a strictly positive cocycle if and
only if $\vpi$ is a local homeomorphism.
\end{corollary}
\begin{proof} By Lemma~\ref{l:locopenadmiss} if the cocycle is
positive at every point of $X$, then $\vpi$ is locally open at every
point, hence it is an open map.  Thus $\vpi$ is a local
homeomorphism.

Conversely, if $\vpi$ is a local homeomorphism, let the sets $X_j$
be as in Corollary~\ref{c:clopen} and set $Z_j = \vpi(X_j).$  Then
the sets $Z_j$ are pairwise disjoint and clopen, and their union is
$X$. If the cocycle $\om$ is defined to be $\frac{1}{j}$ on $X_j,$
then $\om$ is strictly positive.
\end{proof}

Recall that a metric space is \emph{zero dimensional} if it admits a
basis for the topology which is both closed and open.

\begin{proposition} \label{p:zerodim} Let $(\vpi, X)$ be CSLI, and suppose $X$ is
zero-dimensional. Then the necessary condition of
Corollary~\ref{c:locopenadmiss} for a cocycle to exist is also
sufficient.
\end{proposition}

\begin{proof} By compactness, we can obtain a finite cover of
disjoint, clopen sets $\sU_i,\ 1 \leq i \leq m,$ so that for each
$\sU_i,\ \vpi^{-1}(\sU_i)$ is a union of, say $n_i$ disjoint clopen
sets, $\sW_{i, j}$, and the restriction of $\vpi$ to each of them is
injective. We may enumerate them so that $\vpi|\sW_{i, 1}$ is an
open map. Note that since the sets $\sU_i$ are disjoint, we have
that $\sW_{i, j} \cap \sW_{i', j'} = \emptyset $ if $(i, j) \ne (i',
j').$ Thus, the sets $\sW_{i, j}$ constitute a finite, pairwise
disjoint cover of $X$ of clopen sets.

We now define a cocycle $\om$ on $X$ as follows: for each $ i,\ 1
\leq i \leq m$ let $\om(x) = 1 $ for all $x \in \sW_{i, 1}$ and
$\om(x) = 0$ for $x \in \sW_{i, j}$ for $ j > 1.$ Because the sets
where $\om$ is $1$ or $0$ are clopen, $\om$ is continuous.  And the
cocycle condition, that for $y \in X,$
\[ \sum_{\vpi(x) = y} \om(x) = 1 \]
is satisfied, since by construction there is one $x$ for which
$\om(x) = 1,$ and for the remaining $x$ for which $\vpi(x) = y,\
\om(x) = 0.$
\end{proof}

\begin{remark} \label{r:unique} Note that if for every $y \in X$
there exists a unique $x \in \vpi^{-1}(y)$ such that $\vpi$ is
locally open in a neighborhood of $x$, then the cocycle $\om$ is
unique, and hence does not depend on the choice of the cover
$\sU_i.$
\end{remark}

\begin{example} \label{e:notadmiss} This is an example of a semigroup $\sP = \bbN$ of
a CSLI dynamical system which is not admissible.

Let $\vpi: \Pi_{n\in \bbZ} \bbT_n \to \Pi_{n\in \bbZ} \bbT_n $ where
$\bbT_n = \bbT = [0, 1),$ as follows:  for a point $\bold{x} =
(x_n)$ in the product space, set
\[ \vpi(\bold{x}) = \bold{y} \ \text{where } y_{n-1} = x_n \] for all
$ n \neq 1$ and $ y_0 = 2x_1 \ (mod 1).$  Note that $\vpi$ is a
local homeomorphism of $ \Pi_{n\in \bbZ} \bbT_n .$

Let $Z \subset \Pi_{n\in \bbZ} \bbT_n$ consist of those sequences
$\bold{x} = (x_n)$ satisfying: for all $n \geq 1,\ 0 \leq x_n \leq
\frac{1}{2}.$ Clearly $Z$ is closed, and $\vpi(Z) \subset Z.$  We
take
\[ X = \cap_{n=0}^{\infty} \vpi^n(Z) ,\]
where $\vpi^0$ is the identity, and for $n > 1,\ \vpi^n$ is the
$n$-fold composition of $\vpi$ with itself. Then $\vpi(X) = X.$

Changing notation so $\vpi$ refers to the restriction of $\vpi$ to
$X$, the dynamical system $(\vpi, X)$ is CSLI.

To show that $(\vpi, X)$ is not admissible, we suppose to the
contrary that $\om$ is a cocycle for $\vpi.$ Set
\[ \bold{y}^0 = (\dots, 0, \underset{0}{0}, \frac{1}{2},
\frac{1}{2}, \dots)\] where the underset $0$ denotes the $0$-th
position in the array. Note that $\vpi^{-1}(\bold{y}^0) = \{
\bold{y}^0, \bold{w}^0\}$ where
\[ \bold{w}^0 = (\dots, 0, \underset{0}{0}, 0, \frac{1}{2},
\frac{1}{2}, \dots).\] First we show that $\om(\bold{y}^0) = 1.$ To
this end, define
\[ \bold{y}(t) = ( \dots, t, \underset{0}{t}, \frac{t}{2},
\frac{t}{2}, \dots) \] and note that, for $t > 0,\
\vpi^{-1}(\bold{y}(t)) = \{ \bold{y}(t) \}.$ Hence $\om(\bold{y}(t))
= 1$ for $ t > 0.$  Since $\bold{y}(t) \to \bold{y}^0$ as $ t \to 1$
continuity of $\om$ forces $\om(\bold{y}^0) = 1.$

Next we claim that $\om(\bold{w}^0) = 1.$ To this end, we define
\[ \bold{u}(t) = (\dots, t, \underset{0}{t}, \frac{1}{2} - t, \frac{1}{2} -
t, \dots )\] for $ 0 < t < \frac{1}{2}.$ To see that $\bold{u}(t)
\in X,$ note first that $\bold{u}(t) \in Z.$  Let $n \in \bbN, \ n
\geq 1$ and set
\[ \bold{w}(t) = (\dots, t, \underset{0}{t}, \frac{t}{2}, \dots,
\underset{n}{\frac{t}{2}}, \frac{1}{2} - t, \frac{1}{2} - t, \dots
).
\] Then $\bold{w}(t) \in Z$ and $\vpi^n(\bold{w}(t)) = \bold{u}(t).$  Thus, $\bold{u}(t)
\in \vpi^n(Z)$ for every $n \geq 0,$ so $\bold{u}(t) \in X.$ Now set
$n = 1$.  The same argument shows that $\bold{w}(t) \in X, $ and
furthermore $\bold{w}(t)$ is the single inverse image of
$\bold{u}(t).$  Thus, $\om(\bold{w}(t)) = 1.$  As $ t \to 0,\
\bold{w}(t) \to \bold{w}^0.$ Continuity forces $\om(\bold{w}^0) =
1.$  But the cocycle condition $\om(\bold{y}^0) + \om(\bold{w}^0) =
1$ is violated, so no cocycle exists and the system $(\vpi, X)$ is
not admissible, and in particular, $\vpi$ is not a local
homeomorphism.

\end{example}

\begin{example} \label{e:admissCSLI} This is an example of an admissible
CSLI system which is not a local homeomorphism. First we construct
$X$ as follows: let $Z$ be the space obtained from $\bbR$ by
replacing each integer $n \leq 0$ by two points $n^-,\ n^+$ with
$n^- < n^+$ so that $n^-$ is the immediate predecessor of $n^+$. $Z$
is an ordered set which is topologized  by taking as a base for the
open sets all sets of the form $(a, b),\ a < b \in Z$ and $(a, n^-]$
and $[n^+, b)$  for $n,\in \bbZ,\ n  \leq 0$ and $a < n < b.$  Set
$X = Z\cup \{-\infty, +\infty\}$ the two-point compactification of
$Z$.

Note that $X$ is a compact metrizable space.  Define $\vpi: X \to X$
by taking the points $\pm\infty$ to be fixed, and for $ x \neq \pm
\infty$ setting
\begin{equation*}
\vpi(x) =
\begin{cases}
x + 1 \text{ if } x \leq 0^-, \ x \notin \bbZ \\
(n+1)^- \text{ if } x = n^-,\ n \leq -1 \\
(n+1)^+ \text{ if } x = n^+, \ n \leq -1 \\
1 \text{ if } x = 0^-\\
 x \text{ if } x \geq 0^+
\end{cases}
\end{equation*}
Observe that $\vpi$ is CSLI, but
 that $\vpi$ is not a local homeomorphism because it is not an open
 map in a neighborhood of $0^-.$

Notice that the points $y,\ 0^+ \leq y \leq 1$ have two pre-images,
and all other points have one pre-image. Define a cocycle on $X$ by
$\om(x) = 1,\ 0^+ \leq x \leq +\infty,\ \om(x) = 0, \ -1^+ \leq x
\leq 0^-,$ and $ \om(x) = 1,\ -\infty \leq x \leq -1^-.$  Then $\om$
is a cocycle, but not strictly positive.

Note there does not exist a strictly positive cocycle. This follows
from Corollary~\ref{c:strictlypos}, but can also be seen directly.
For $y = 1$ has two pre-images, namely $0^- $ and $1$, but any point
$ y
> 1$ has only one preimage, $\vpi^{-1}(y) = y$ so that for such $y$
necessarily $\om(y) = 1.$ Continuity of $\om$ forces $\om(1) = 1,$
hence $\om(0^-) = 0.$
\end{example}

\begin{remark} Define the conditional expectation
\[ E(f)(x) = \al \circ \sL(f)(x) = \sum_{\vpi(u) = \vpi(x)}
\om(u)f(u)\] where $\al(g) = g\circ \vpi,\ g \in C(X).$

Then if $\om$ is not strictly positive, the conditional expectation
can be degenerate. Indeed, suppose
 $\om(x) = 0$ in a neighborhood $\sU$ of a point
$x_0$.  Suppose $f$ is a nonnegative function supported in $\sU$ and
that $\vpi$ is injective on $\sU.$ Then for $x \in X$
\begin{align*}
E(f)(x) &= \sum_{\vpi(t) = \vpi(x)} f(t)\om(t) \\
    &= 0
\end{align*}
since  $\om$ is zero where $f$ is nonzero.

Thus, the conditional expectation associated to the cocycle in
Example~\ref{e:admissCSLI} is degenerate.  However, it is possible
to define a cocycle on the space $X$ in that example so that it is
nondegenerate, as we show.
\end{remark}

\begin{example} \label{e:admissnondeg} Let $(\vpi, X)$ be as in
Example~\ref{e:admissCSLI}. Define a cocycle $\om$ as follows:
$\om(0^-) = \om(0^+) = 0. $ For $ 0^+ < x < 1,$ set $\om(x) = x$ and
$\om(x-1) = 1-x.$ As before, for $ x \geq 1$ we are forced to have
$\om(x) = 1.$ Also as before, we must have $\om(x) = 1$ for $x \leq
-1^-.$ Since $\vpi^{-1}(0^+) = \{-1^+, 0^+\}$ and $\om(0^+) = 0,$
necessarily $\om(-1^+) = 1.$ The cocycle vanishes on the set
$\{0^-,\ 0^+\},$ which has no interior. Thus the resulting
expectation $E$ is nondegenerate.
\end{example}

\begin{example} \label{e:notadmiss2} We take $X$ as in Example~\ref{e:admissCSLI}
and define $\vpi$ as it is there for $x \leq 0^-.$  For $x \geq 0^+$
we define $\vpi$ by: $\vpi(0^+) = 1$ and $\vpi(x) = x+1$ for $ x >
0^+.$ Then the only point with more than one inverse image is $1$,
and $\vpi^{-1}(1) = \{0^-, \ 0^+\}.$  If the system were admissible,
then necessarily $\om(x) = 1$ for all $x \in X,\ x \neq 0^-,\ 0^+.$
Then continuity would force $\om(0^-) = 1 = \om(0^+),$  violating
cocycle property (ii).  Thus this provides another example of a non
admissible system.
\end{example}

\subsection{Finitely generated free semigroups}
Next we want to consider finitely generated free semigroups; that
is, semigroups isomorphic to $(\bbZ^+)^k,$ where $\bbZ^+ = \bbN \cup
\{0\}.$ Recall that (cf. \cite{Fu}) in an abelian group $\sG$ a
finite set of elements $\{ a_1, \dots, a_k\}$ with $a_i \neq a_j$
for $ i \neq j$ is \emph{independent} if any relation of the form
\[ n_1 a_1 + \cdots + n_k a_k = 0, \quad (n_i \in \bbZ) \]
implies
\[ n_1 = \cdots = n_k = 0.\]

In our case, we are dealing with elements of an abelian semigroup,
not a group. We could of course take recourse to the fact that the
semigroup is embedded in a (smallest) abelian group, and make use of
this definition in the ambient group.  However, the semigroup is
represented by maps which may not be invertible, and an ambient
group is quite removed from the context of the dynamical system, so
it is natural to want to express independence in the context of the
semigroup. As it happens, it is easy to recast the definition of
independence in the semigroup context.

\begin{definition} Let $a_1, \dots, a_k$ be elements of an abelian
semigroup $\sP.$ The set $\{a_1, \dots, a_k\}$ with $a_i \neq a_j$
for $i \neq j$ will be called \emph{independent} if for any nonempty
subset $E \subset \{ 1, \dots, k\}$ and nonnegative integers $n_1,
\dots, n_k$ the relation
\[ \sum_{j \in E} n_j a_j = \sum_{j \in E^c} n_j a_j \]
implies
\[ n_1 = \cdots = n_k = 0.\]
In case the complement $E^c = \emptyset,$ we interpret the right
side of the equation to be zero.
\end{definition}

Let $\sP$ be an abelian semigroup with $0$, and let $\{a_1, \dots,
a_k\}$ be a set of independent generators of $\sP.$
\begin{proposition} \label{p:freesemigroup}
Let $\sP$ act on the compact metric space $X$. Then the the action
is admissible iff each $\vpi_{a_j}$ is an admissible action, $ 1
\leq j \leq k,$ and
\[ \om_i(1, x)\om_j(1, \vpi_i(x)) = \om_j(1, x)\om_i(1, \vpi_j(x))
\quad \ddag \] for all $i, j \in \{1, \dots, k\}$, where $\om_i$ is
the cocycle associated with the subsemigroup $\bbZ^+\,a_j,\ 1 \leq j
\leq k,$ and we have written $\vpi_j$ for $\vpi_{a_j}.$
\end{proposition}

\begin{proof} Note that $\sP$ is
isomorphic with $(\bbZ^+)^k$ under the isomorphism $a_j \to e_j$,
where $e_j$ is the standard basis element
$(0,\dots,0,\underset{j}{1},0,\dots,0).$ For convenience, we work
with $(\bbZ^+)^k$ in place of $\sP.$

Suppose $\sP$ is admissible with cocycle $\om$.  If we set
$\om_i(\ell, x) = \om(\ell\,e_i, x),$ then the cocycle identity
shows that the condition ($\ddag$) is necessary.

Conversely, suppose cocycles $\om_i$ are given which satisfy the
conditions of the proposition.  For $m \in (\bbZ^+)^k,\ m = (m_1,
\dots, m_k),$ let $|m| = m_1 +\cdots + m_k.$ We will define a
cocycle $\om(m, x)$ on $\sP$ by induction on $|m|.$

For $|m| = 0,$ set $\om(m, x) = 1.$ For $|m| = 1,\ m = e_j$ for some
$j$ and we define $\om(m, x) = \om_j(1, x).$ For $|m| = 2,$ either
$m = 2e_j$ for some $j$ or else $m = e_i + e_j$ for some $i \neq j.$
In the first case, set $\om(m, x) = \om_j(2, x).$ This satisfies the
cocycle identity because $\om_j$ does. In the second case, define
$\om(m, x) = \om_i(1, x)\om_j(1, \vpi_i(x)).$ Note that the cocycle
identity is satisfied by assumption ($\ddag$).

Suppose now that $\om(m, x)$ has been defined and satisfies the
cocycle identity for $|m| \leq N$ for $N > 1.$ We now define $\om(m,
x)$ for $|m| = N+1.$

Let $m,\ n, \ p, \ q \in \sP$ be such that $n + m = p + q$ and $|n|,
|m|, |p|, |q|$ are all positive, and $|n|+|m| = |p|+|q| = N+1.$ We
claim that
\[ \om(n, x)\om(m, \vpi_n(x)) = \om(p, x)\om(q, \vpi_p(x)).\]
We do this first assuming $|n-p| = |m-q| = 1.$ Thus, either there
exists $j$ such that $n = p + e_j$ or $p = n + e_j.$  The two cases
are similar; we do the first case. Then, $ q = m + e_j.$ Thus,
\begin{align*}
\om(n, x)\om(m, \vpi_n(x)) &= \om(p+e_j, x)\om(m, \vpi_n(x)) \\
    &= \om(p, x)\om(e_j, \vpi_p(x))\om(m, \vpi_n(x)) \\
    &= \om(p, x)\om(e_j, \vpi_p(x))\om(m, \vpi_{e_j+p}(x)) \\
    &= \om(p, x)\om(e_j, \vpi_p(x))\om(m, \vpi_{e_j}(\vpi_p(x))) \\
    &= \om(p, x)\om(m+e_j, \vpi_p(x)) \\
    &= \om(p, x)\om(q, \vpi_p(x))
\end{align*}
where we have used the induction hypothesis that the (partially
defined) cocycle satisfies the cocycle identity where it is already
defined. For the case where $|n-p| > 1,$ we repeat the first step
$|n-p|$ times.

Thus if $m+n = N+1,$ if we set $\om(n+m, x) = \om(n, x)\,\om(m,
\vpi_n(x)),$ the above calculation shows that $\om$ is a well
defined cocycle which satisfies the cocycle condition. Conditions
$(1)$ through $(3)$ of  Definition~\ref{d:cocycle} are easily
verified.
\end{proof}

\section{Divisible semigroups}
An abelian group $\sG$ is \emph{divisible} if the equation $ mx = a
\ (m \in \bbN,\ a \in \sG)$ has a solution $x \in \sG$ (\cite{Fu}).
One could use the same definition for semigroups. However, we want
to consider examples such as the semigroup $\sP$ of positive dyadic
rationals. Let $\sD = \{ \frac{k}{2^n}, \ k \in \bbZ, \ n \in \bbZ
\}$ be the group of dyadic rationals.  This is not divisible, as the
equation $mx = a$ is solvable for $ x \in \sD$ only for $ m $ a
power of $2$. Thus, for our purposes an alternative definition is
appropriate.

\begin{definition} \label{d:divisible}
A sequence $\{ d_k\}$ in a semigroup $\sP$ will be called a
\emph{fundamental sequence} if
\begin{enumerate}
\item  there exists a sequence of integers $n_k > 1$ such that $d_k =
n_k d_{k+1},\ k \geq 1,$ and
\item for every $d \in \sP$ there exists $k \in \bbN$ such that $d_k$
divides $d$.
\end{enumerate}
We say $\sP$ is \emph{divisible} if it contains a fundamental
sequence.
\end{definition}

\begin{proposition} \label{p:divisible} Let $\sP$ be a divisible semigroup of CSLI maps
 on $X$. Then either all $\vpi_d,\ d \in \sP,$ are homeomorphisms,
or else none is a homeomorphism.
\end{proposition}

\begin{proof}
Suppose for some $d \in \sP, \ \vpi_d$ is a homeomorphism. Let $ e$
be another element of $\sP.$

If $d_k$ divides $d$ and $d_{\ell}$ divides $e$, taking $n = \max\{
k,\ \ell\}$ we have that $d_n$ divides both $d,\ e.$ Say $ d =
md_n$; then $\vpi_d$ is the $m-$ fold composition of $\vpi_{d_n}$
with itself.  Since the composition is injective, $\vpi_{d_n}$ is
injective, hence $\vpi_{d_n}$ is a homeomorphism. But since $ e =
m'd_n,$ for some $m' \in \bbN,\ \vpi_e$ is a composition of
homeomorphisms, hence is a homeomorphism.
\end{proof}

\begin{proposition} \label{p:divis2} Suppose $\sP$ is a divisible
semigroup of CSLI maps which are not homeomorphisms.  Then there is
an $x_0 \in X$ satisfying
\[ |\vpi_d^{-1}(x_0)| > 1 \]
for all $d \in \sP .$ Furthermore there exist $u_0 \neq v_0$ such
that for all $d \in \sP, \ \vpi_d(u_0) = \vpi_d(v_0).$
\end{proposition}

\begin{proof}
By Proposition~\ref{p:divisible}, either all the maps $\vpi_d\ (d\in
\sP)$ are homeomorphisms or none is. Let $\{d_k\}$ be a fundamental
sequence, and for each $k$ let $x_k$ satisfy $|\vpi_{d_k}^{-1}(x_k)|
> 1.$ Taking a subsequence, we may assume $x_k \to x_0.$ Now if for
some $d \in \sP $ we had that $|\vpi_d^{-1}(x_0)| = 1,$ then by
Corollary~\ref{c:cardopen} there is a neighborhood $\sU$ of $x_0$
such that $ |\vpi_d^{-1}(x)| = 1 $ for all $x \in \sU$. Let $d_k$ be
such that $d_k$ divides $d$, say $d = m d_k,$ with $k$ sufficiently
large so that $x_k \in \sU.$ But then
\[ |\vpi_d^{-1}(x_k)| \geq |\vpi_{d_k}^{-1}(x_k)| > 1.\]
Thus, $|\vpi_d^{-1}(x_0)| > 1.$
 This proves the first assertion.

Next, let $\{ u_k\},\ \{v_k\}$ be sequences such that
\[ u_k \neq v_k \text{ and } \vpi_{d_k}(u_k) = \vpi_{d_k}(v_k) =
x_0.\] By taking subsequences, we may assume that $u_k \to u_0,\ v_k
\to v_0.$

Fix $d \in \sP$ and write $d = d_k + e_k$, for $k \geq
 N $ for some $N \in \bbN.$ Then
\begin{align*}
\vpi_d(u_k) &= \vpi_{e_k}\circ \vpi_{d_k}(u_k) \\
    &= \vpi_{e_k}(x_0).
\end{align*}
Now by taking a subsequence we may assume $\vpi_{e_k}(x_0) \to y_0,$
say. Since $u_k \to u_0,\ \vpi_d(u_k) \to \vpi_d(u_0).$ But by the
above, $\vpi_d(u_k) \to y_0.$ Thus $\vpi_d(u_0) = y_0.$  Similarly,
$\vpi_d(v_0) = y_0.$

Now if $u_0 = v_0,$ then $\vpi_d$ is not injective in any
neighborhood of $u_0.$  Thus, $u_0 \neq v_0.$
\end{proof}

Thus we can state our
\begin{theorem}  Let $\sP$ be a divisible semigroup of
CSLI maps acting on a compact metric space $X$. Suppose $\sP $
separates the points of $X$.  Then $\sP$ consists of homeomorphisms.
\end{theorem}
\begin{proof} By Proposition~\ref{p:divisible}, either $\sP$ consists
of homeomorphisms, or else of CSLI maps which are not
homeomorphisms. Suppose the latter is the case. Then with $u_0,\
v_0$ as in Proposition~\ref{p:divis2}, let $d \in \sP.$ By the
Proposition, $u_0,\ v_0$ map to the same point under $\vpi_d,$ for
all $d \in \sP.$ But that contradicts the assumption that $\sP$
separates the points of $X$. Thus the CSLI maps must in fact be
homeomorphisms.
\end{proof}

It is not \textit{a priori} obvious that divisible semigroups of
CSLI maps which are not homeomorphisms exist. Before constructing
the example, we remind the reader of a construction which has been
used to ``cut up'' the real numbers to obtain a zero-dimensional
space, $Z$. For each dyadic rational $ d \in \bbR, $ we replace $d$
by two points $d^-$ and $d^+$ so that $d^- < d^+$ and no point lies
between $d^-, \ d^+.$ Thus $Z$ is an ordered set. Now we introduce a
topology by taking as a base $\sB$ for the topology the sets $ [r^+,
s^-]$ where $r < s$ are dyadic rationals. In this topology, every
``open interval'' $(a, b) = \{ x \in Z,\ a < x < b \}$ is an open
set in the topology of $Z$.

Observe that the complement of an interval $[r^+, s^-]$ is also
open, so that $[r^+, s^-]$ is closed, hence clopen. One can show
that the closed intervals $ [a, b],\ a < b \in Z, $ are compact.
Thus $Z$ is a locally compact Hausdorff space, which is metrizable,
as the base $\sB$ is countable.

\begin{example} \label{e:divisadmiss} This is an example of a divisible semigroup.
We construct a compact metric space $X = X_1 \cup X_2 \cup X_3,$ the
union of three disjoint sets.  Take
\[ X_1 = [0^+, +\infty] \] the one-point compactification of the
interval $z \in Z:\ z \geq 0^+$. Now we take $X_2,\ X_3$ both to be
the one-point compactifications of copies of $(-\infty, 0^-] $ in
$Z.$ To distinguish them, we use superscripts hat and tilde. Thus,
\[ X_2 = [-\hat{\infty}, \hat{0}^-] \text{ and } X_3 =
[-\tilde{\infty}, \tilde{0}^-] .\]

The set $Z$ is not a group under addition, but there is an action of
the group $\sD$ of dyadic rationals on $Z,$ as follows: let $d \in
\sD$ and define translation $\vpi_d$ on $Z$ by
\begin{equation*}
\vpi_d(x) =
\begin{cases}
d+x \text{ if } x \text{ is not a dyadic rational}; \\
(d + x)^+ \text{ if } x = r^+ \text{ where } r \text{ is a dyadic rational}\\
(d + x)^- \text{ if } x = r^- \text{ where } r \text{ is a dyadic
rational}.
\end{cases}
\end{equation*}
It is easy to see that $\vpi_d$ is continuous, as $\vpi_d^{-1}$ maps
basic open intervals to basic open intervals. Similarly, $\vpi_d$ is
seen to be an open map. Since it is both injective and surjective,
it is a homeomorphism of $Z$.

Let $\sP$ denote the positive dyadic rationals, and, changing
notation, let $\vpi_d \ (d \in \sP)$ denote an action of $X$, which
we define as follows: $\vpi_d$ leaves all three points at infinity
fixed. For $x \in X_1 \cup X_2$ not a point at infinity, we let
$\vpi_d(x) $ be defined as follows:  if $x \in X_2, \ x = \hat{z}$
for some $z \in Z,$  and $ d \in \sP,$
\begin{equation*}
\vpi_d(x) =
\begin{cases}
\widehat{d + z} \in X_2 \text{ if } d + z \leq 0^- \\
d + z \in X_1 \text{ if } d + z \geq 0^+.
\end{cases}
\end{equation*}
If $x \in X_1,$ then $\vpi_d(x)$ is defined exactly as on $Z$.
$\vpi_d$ acts similarly on $X_1 \cup X_3.$ Clearly, $\vpi_d$ is
surjective on $X$. And in the same way as with $Z$, one sees that
$\vpi_d$ is continuous and open. Note that if $ x \in [-d^+, 0^-]
\subset Z$ that $\vpi_d(\hat{x}) = \vpi_d(\tilde{x})$, so that
$\vpi_d$ is not one-to-one.  Thus, $\{ \vpi_d:\ d \in \sP \}$ is a
semigroup of local homeomorphisms on the compact space $X$ which are
not homeomorphisms.  Note that the semigroup $\sP$ has a fundamental
sequence, namely $\{ \frac{1}{2^n} \}_{n \in\bbN}$, so that $\sP$ is
a divisible semigroup.

Next we show that the semigroup is admissible. To simplify notation,
when $x \in X$ belongs to either $X_2$ or $X_3$, and there is no
need to distinguish between $X_2,\ X_3,$ we will omit the
superscripts hat and tilde. Define the cocycle $\om$ on $\sP \times
X$ by
\begin{equation}
\om(d, x) =
\begin{cases}
1, \text{ if } x \leq -d^- \text{ or } x \geq 0^+ \\
\frac{1}{2} \text{ if } -d^+ \leq x \leq 0^- .
\end{cases}
\end{equation}
We do this for all $d \in \sP$ and $x \in X$.  We need to show this
is consistent with the cocycle identity.  So suppose $ e, \ f \in
\sP $ with $ e + f = d.$ Suppose $ x \leq -d^-.$  Then $\om(e, x) =
1.$ Now we claim that $\vpi_e(x) \leq -f^-. $ For otherwise, we
would have $\vpi_e(x) \geq -f^+, $ hence $ \vpi_f(\vpi_e(x)) \geq
\vpi_f(-f^+) \geq 0^+.$  But then $\vpi_d(x) = \vpi_f(\vpi_e(x))
\geq 0^+$ which contradicts that $ x \leq -d^-.$ Thus, both $\om(e,
x)$ and $ \om(f, \vpi_e(x)) $ equal $1$, as does $ \om(d, x).$

The case were $x \geq 0^+$ is easier, for then it is clear that
$\om(e, x)$ and $ \om(f, \vpi_e(x)),$ and $\om(d, x)$ are all equal
to $1$.

Finally, let $-d^+ \leq x \leq 0^-.$  Consider two cases: if
$\vpi_e(x) \geq 0^+, $ then $|\vpi_e^{-1}(\vpi_e(x))| =2,$ so that $
\om(e, x) = \frac{1}{2}.$  As $\vpi_e(x) \geq 0^+, \ \om(f,
\vpi_e(x)) = 1.$ By definition $\om(d, x) = \frac{1}{2}, $ so that
the equality
\[ \om(x, d) = \om(e, x) \om(f, \vpi_e(x)) \]
holds. In the other case, $ \vpi_e(x) \leq 0^-,$ that is,
\[ -d^+ \leq x \leq \vpi_e(x) \leq 0^- .\]
Then $\om(e, x) = 1.$  But $\vpi_f(\vpi_e(x)) = \vpi_d(x) \geq 0^+$
so that both $\om(f, \vpi_e(x)),$ and $ \om(d, x) = \frac{1}{2}.$ So
again the cocyle identity
\[ \om(d, x) = \om(e, x) \om(f, \vpi_e(x)) \]
holds.

Finally we need to observe that for arbitrary $ d \in \sP,$ the map
\[ x \to \om(d, x) \]
is continuous.  But observe that
 \[ \{ x:\ \om(d, x) = \frac{1}{2}\} = [-\hat{d}^+, \hat{0}^-] \cup [-\tilde{d}^+,
 \tilde{0}^-] \]
which is a clopen set. Thus, the set where the cocycle is $1$ is
also clopen, and so the cocycle is continuous.

\end{example}

\begin{remark} Proposition~\ref{p:divis2} shows that some
features of Example~\ref{e:divisadmiss} are not arbitrary. The role
of the point $x_0$ in the Proposition is played by $0^+$, and the
roles of $u_0,\ v_0 $ are played by $\hat{0}^-,\ \tilde{0}^-$ in the
example.
\end{remark}

\begin{remark} We do not know of an example of a divisible
semigroup of CSLI maps where the maps are not local homeomorphisms.
\end{remark}


\begin{thebibliography}{99}

\bibitem{Ba} V. Baladi, \textit{Positive transfer operators and decay
of correlations} in Advanced Series in Nonlinear Dynamics, 16, World
Scientific Publishing Co., Inc., River Edge, NJ, 2000.

\bibitem{ER} R. Exel and J. Renault, \textit{Semigroups of local
homeomorphisms and interaction groups}, Ergodic Theory Dynam.
Systems \textbf{27}, 2007, no. 6, 1737--1771.

\bibitem{Fu} L. Fuchs, \emph{Abelian Groups}, Publishing House of the Hungarian Academy of Sciences, Budapest, 1958.


\end{thebibliography}
\end{document}